\documentclass[12pt,a4paper]{amsart}
\usepackage{
    amsmath,  amsfonts, amssymb,  amsthm,   amscd,
    gensymb,  graphicx, comment,  etoolbox, url,
    booktabs, stackrel, mathtools,enumitem, mathdots,  microtype, lmodern,    mathrsfs, graphicx, tikz,  longtable,tabularx, float, tikz, pst-node, tikz-cd, multirow, tabularx, amscd,  bm, array, makecell, diagbox, booktabs,ragged2e,  footmisc
}
\usepackage[utf8]{inputenc}
\usepackage{microtype, fullpage, wrapfig,textcomp,mathrsfs,csquotes,fbb}
\usepackage[colorlinks=true, linkcolor=blue, citecolor=blue, urlcolor=blue, breaklinks=true]{hyperref}
\usepackage[capitalise]{cleveref}
\usepackage{orcidlink}
\usepackage{todonotes}
\usetikzlibrary{arrows}

\newcommand{\rvline}{\hspace*{-\arraycolsep}\vline\hspace*{-\arraycolsep}}
\newtheorem{theorem}{Theorem}[section]

\newtheorem{definition}[theorem]{Definition}

\newtheorem{lemma}[theorem]{Lemma}

\newtheorem{proposition}[theorem]{Proposition}
\newtheorem{remark}[theorem]{Remark}

\newcommand\Q{\mathbb{Q}}

\newcommand{\la}{\langle}
\newcommand{\ra}{\rangle}
\newcommand{\malpha}{\mathrm{M}_{\alpha}}
\newcommand{\mbalpha}{\overline{\mathrm{M}}_{\alpha}}

\begin{document}
\begin{center}
{ 
       {\Large \textbf{\sc Real Bott manifold structure of $n$-dimensional Klein bottle and its rational Betti numbers }}
\\

\medskip

 {\sc Navnath Daundkar\footnote{Corresponding author}}\\
{\footnotesize Indian Institute of Science Education and Research Pune, India}\\

{\footnotesize e-mail: {\it navnath.daundkar@acads.iiserpune.ac.in}}}\\

 {\sc Priyavrat Deshpande}\\
{\footnotesize Chennai Mathematical Institute, India}\\

{\footnotesize e-mail: {\it pdeshpande@cmi.ac.in}}\\

\end{center}

\medskip

\begin{center}
  {\sc Abstract}\\
\end{center}
Donald Davis initiated the study of an $n$-dimensional analogue of the Klein bottle. This generalized Klein bottle occurs as a moduli space of planar polygons for a certain choice of side lengths.
In this paper, we show that the $n$-dimensional Klein bottle is a real Bott manifold and determine the corresponding Bott matrix. We determine the small cover structure on two other classes of moduli spaces of planar polygons. As an application, we compute the rational Betti numbers of these spaces using a formula, due to Suciu and Trevisan.

\hrulefill


\vspace{0.3cm}

\section{Introduction}
The \emph{moduli space of planar polygons} (or \emph{planar polygon space}) associated with a \emph{length vector} $\alpha:=(\alpha_{1},\dots, \alpha_{n+3})$, denoted by $\overline{\mathrm{M}}_{\alpha}$, is the collection of all closed piecewise linear paths with side lengths $\alpha_{1}, \alpha_{2},\dots, \alpha_{n+3}$ in the plane viewed up to all isometries. 
In other words, %we have
\[\overline{\mathrm{M}}_{\alpha}\coloneqq \big\{(v_{1},v_{2},\dots,v_{n+3})\in (S^{1})^{n+3} : \displaystyle\sum_{i=1}^{n+3}\alpha_{i}v_{i} = 0 \big\}\big/\mathrm{O}_{2},\]
where $S^{1}$ is the unit circle and the group of isometries $\mathrm{O}_{2}$ acts diagonally. 
If we consider the collection of closed piecewise linear paths in the plane upto orientation preserving isometries (i.e. $\mathrm{SO}_2$), then the corresponding quotient space is denoted by $\malpha$. Note that $\malpha$ is a double cover of $\mbalpha$.
It was shown in \cite[Theorem 1.3]{FAR} that, if we choose a length vector $\alpha$ such that $\sum_{i=1}^{n+3}\pm \alpha_{i} \neq 0$ then  $\overline{\mathrm{M}}_{\alpha}$ and $\malpha$ are closed, smooth manifolds of dimension $n$. Such length vectors are called \emph{generic} length vectors. 

Various topological aspects of these spaces have been studied. For example,
in \cite{Kamiyamahomology1}, Kamiyama and Tezuka  proved that for a length vector $\alpha=(1,\dots,1,r)$, the integral homology of $\malpha$ is torsion-free and computed the Betti numbers. 
Expressing various topological invariants of planar polygon spaces in terms of combinatorial data associated with the corresponding length vector is one of the important tasks in this area. 
Farber and Sch\"{u}tz \cite{FS} showed that for an arbitrary length vector, the integral homology groups of $\mathrm{M}_{\alpha}$ are torsion-free and also described the Betti numbers of $\malpha$ in terms of the combinatorial data associated with the length vector, called short subsets. 

\begin{definition}
Let $\alpha=(\alpha_{1},\alpha_{2},\dots,\alpha_{n+3})$ be a length vector.
A subset $I\subset [n+3]$ is short  if 
\[\sum_{i\in I} \alpha_i  < \sum_{j \not \in I} \alpha_j.\]  A subset is long if its complement is short.  
\end{definition}
 
In general, the collection of short subsets may be very large. 
Hence, there is another combinatorial object that efficiently encodes the information about all short subsets (it was introduced by Hausmann \cite[Section 1.5]{hausmann2007geometric}). 
Since the diffeomorphism type of a planar polygon space does not depend on the ordering of the side lengths, we assume $\alpha_{1}\leq \alpha_{2} \leq \dots\leq \alpha_{n+3}$.
For a (generic) length vector $\alpha$, consider the following collection of subsets of $[n+3]$:
\[ S_{n+3}(\alpha) = \{J\subset [n+3] : \text{ $n+3\in J$ and $J$ is short}\}. \]
There is a partial order on these subsets given by $I\leq J$ iff $I=\{i_{1},\dots,i_{t}\}$ and $\{j_{i},\dots,j_{t}\}\subseteq J$ with $i_{s}\leq j_{s}$ for $1\leq s\leq t$. 
\begin{definition}\label{gc}
The genetic code of a length vector $\alpha$ is the set of maximal elements of $S_{n+3}(\alpha)$ with respect to the above partial order. 
The maximal elements are called genes.
\end{definition}

If $A_{1}, A_{2},\dots, A_{k}$  are the maximal elements of $S_{n+3}(\alpha)$ with respect to $\leq$ then the genetic code of $\alpha$ is denoted by $\la A_{1},\dots,A_{k}\ra$.
It follows from \cite[Lemma 1.2]{hausmann2007geometric} that the genetic code of a length vector $\alpha$ determines the diffeomorphism type of the manifold $\malpha$.
For the length vector $\alpha=(1,1,\dots,1)$, Kamiyama \cite{Kamiyamahomology2} determined the homology groups $H_{\ast}(\overline{\mathrm{M}}_{\alpha},\mathbb{Z}_{p})$ for odd primes and $H_{\ast}(\overline{\mathrm{M}}_{\alpha},\mathbb{Q})$. 
Hausmann and Knutson \cite{HK1} computed the $\mathbb{Z}_{2}$-cohomology ring of $\overline{\mathrm{M}}_{\alpha}$ and showed that it can be completely determined by the genetic code of the length vector (see \cite[Theorem 2.1]{Daviscring} for more detailed expression).

In this paper, we are interested in studying a specific class of moduli spaces of planar polygons through the lens of toric topology.
Such attempts previously have occurred in a paper by Hausmann and Knutson \cite{HK1} to compute the mod-$2$ cohomology ring of planar polygon spaces.
They also showed in \cite[Section 6]{polspacesGrassmann} that $2$ and $3$-dimensional spatial polygon spaces are, in fact, toric manifolds. Hausmann and Rodriguez \cite[Proposition 6.8]{Hausmann-Rodriguez} obtained a sufficient condition for spatial polygon spaces to be toric manifolds.

One of our aims is to determine the rational Betti numbers of this class using tools from toric topology. This class of planar polygon spaces generalizes the classical Klein bottle.  
An $n$-dimensional analogue of the Klein bottle, denoted $K_{n}$, was introduced by Davis \cite{zbMATH07134923} as follows:
\begin{equation}\label{eq}
     K_{n}=\frac{(S^{1})^{n}}{(z_{1},\dots,z_{n-1},z_{n})\sim( \bar{z}_{1},\dots,\bar{z}_{n-1},-z_{n})}.
\end{equation}
The circle $S^1$ is considered as the unit circle in $\mathbb{C}$ and $\bar{z}$ is the complex conjugate. 
It is easy to see that $K_2$ is the Klein bottle. Donald Davis computed various invariants of the $n$-dimensional Klein bottle. For example, the fundamental group, integral cohomology algebra and the stable homotopy type of $K_{n}$ have been computed. We note that the Betti numbers of $K_n$ can be obtained using \cite[Theorem 2.6]{zbMATH07134923}. 

The following result is an immediate consequence of \cite[Proposition 2.1]{hausmann2007geometric} that justifies the connection with planar polygon spaces. 
\begin{theorem}
Let $\alpha$ be a length vector with the genetic code  $\la\{1,2,\dots,n-1,n+3\}\ra$. Then 
$\overline{\mathrm{M}}_{\alpha}\cong K_n.$
\end{theorem}

In \cite{daundkar2021moment}, we have studied the question of classifying aspherical planar polygon spaces. More precisely, we classified aspherical chain spaces (fixed points of a certain involution on abelian polygon spaces) and the classification is completely determined by the combinatorial data associated with the length vector, called a short code (see \cite[Definition 3.1] {daundkar2021moment}). We note that the definition of short code is inspired by the genetic code.  As observed in \cite{HK1}, and also follows from \cite[Proposition 3.3]{daundkar2021moment}, chain spaces form a subclass of planar polygon spaces. Since these spaces possess a real toric manifold structure, we could use tools from toric topology to classify aspherical chain spaces. In this paper, we focus on computing rational Betti numbers of $n$-dimensional Klein bottle and two other classes of planar polygon spaces.

Real Bott manifolds are an important class of manifolds in toric topology. One of our results (see \Cref{rbm}) shows that $K_n$ is a real Bott manifold, and this helps us to make use of Suciu and Trevisan's (\Cref{thm:ST}) formula to compute the rational Betti numbers of $K_n$.  We would like to mention that the rational Betti numbers of $n$-dimensional Klein bottle can also be computed using Donald Davis's result \cite[Theorem 2.6]{zbMATH07134923}. On the other hand, our computations use the tools from toric topology.
Using real Bott manifold structure, we conclude that $K_n$ can never be cohomologically symplectic (see \Cref{prop:cohosymplect}, and we obtain the description of its rational cohomology ring.

The article is organized as follows: In \Cref{sec2}, we recall some basic notions from toric topology related to small covers and real toric manifolds. We show that $n$-dimensional Klein bottle is a real toric manifold and compute the corresponding Bott matrix (see \Cref{rbm}). We then use the Suciu-Trevisan formula to compute rational Betti numbers of $K_n$ (see \Cref{thm:QBettiKn}). In \Cref{sec3}, we obtain the small cover structure on moduli spaces of planar polygons corresponding to the genetic codes $\la\{1,2,\dots,n-2,n,n+3\}\ra$ and $\la\{1,2,\dots,n-2,n+1,n+3\}\ra$ respectively (see \Cref{thm:small-cover-sructure-P5I} and \Cref{thm:small-cover-sructure-P6I}). As an application, we compute their rational Betti numbers (see \Cref{thm:QbettiP5I} and \Cref{thm:QbettiP6I}).

\section{Realizing \texorpdfstring{$K_{n}$}{Lg} as the real Bott manifold and its rational Betti numbers \label{sec2}}

In this section, we show that $K_{n}$ is a real Bott manifold and determine the corresponding Bott matrix.
Since real Bott manifolds are also examples of small covers (topological analogues of real toric varieties), we begin the section by defining characteristic functions.
Then, we define such a function on the $n$-dimensional cube and show that the corresponding small cover is homeomorphic to the $n$-dimensional Klein bottle.

Let $P$ be an $n$-dimensional simple polytope and  $\mathcal{F} = \{F_1,\dots, F_m\}$ be the collection of its facets. 

\begin{definition}
A function $\chi:\mathcal{F}\to \mathbb{Z}_2^n$ is called characteristic for $P$ if for each vertex $v= F_{i_1}\cap\cdots\cap F_{i_n}$, the $n\times n$ matrix whose columns are $\chi(F_{i_1}),\dots,\chi(F_{i_n})$ is invertible.
Equivalently, we can think of the characteristic function as an $n\times m$-matrix of $0$'s and $1$'s \
\[ \chi =
\begin{bmatrix}
  \begin{matrix}
  \chi(F_{1}) & \chi(F_{2}) & \dots & \chi(F_{n})
  \end{matrix}
\end{bmatrix}
\] with the above property satisfied. 
\end{definition}

An $n$-dimensional \emph{small cover} $M$ is a closed, smooth manifold with an action of $\mathbb{Z}_2^n$ that is locally isomorphic to the standard action of $\mathbb{Z}_2^n$ on $\mathbb{R}^n$ and such that the orbit space is an $n$-dimensional simple polytope $P$.
These manifolds are topological generalizations of real toric varieties. 
M. Davis and T. Januszkiewicz, in their seminal work, showed how to build a small cover from the quotient polytope (see \cite[Section 1.5]{zbMATH04212906} for details). 
Their result states that the manifold has a regular cell structure consisting of $2^n$ copies of the quotient polytope as the top-dimensional cells. 
Here is a brief description.
Given a pair $(P, \chi)$ of a simple polytope and a characteristic function defined on its facets, the corresponding small cover $X(P, \chi)$ is constructed as follows:
\[X(P, \chi) := \left((\mathbb{Z}_2)^n\times P\right)/ \{(t,p) \sim (u,q) \}\quad \text{if~} p = q \text{~and~} t^{-1}u \in \mathrm{stab}(F_q)\]
where $F_q$ is the unique face of $P$ containing $q$ in its relative interior. 
\smallskip

The $n$-dimensional cube is given by 
\[I^{n}=[-1,1]^n=\{(x_{1},\dots,x_{n})\in \mathbb{R}^n : -1\leq x_{i}\leq 1 \text{ for } 1\leq i\leq n\}.\]
Consider the following labelling of the facets of $I^{n}$. 
For each  $1\leq i \leq n$, we set 
\[F_{i}=I \times \dots \times \{-1\}\times \dots \times I ~~ \text{and}  ~~ F_{n+i}=I \times \dots \times \{1\}\times \dots \times I, \] where $\{-1\}$  and $\{1\}$ is at the $i$-th position.
%Let $\mathcal{F}(I^{n})$ be the collection of facets of $I^{n}$.
Define %a function \[\chi : \mathcal{F}(I^{n})\longrightarrow \mathbb{Z}_{2}^{n}\] as
\begin{equation}\label{eq:charforKleinbottle}
\chi(F) = \begin{cases}
e_{i} & \text{if $F=F_{i}$ or $F=F_{n+i}$, $2\leq i \leq n$ ,}\\
e_{1} & \text{if $F=F_{1}$,}\\
\sum_{i=1}^{n}e_{i} & \text{if $F=F_{n+1}$}.\\
\end{cases}    
\end{equation}

One can note that the $n\times 2n$-matrix of $\chi$ can be given as 
\begin{equation}\label{eq2}
    \chi = 
\begin{bmatrix}
1 & 0 & \cdots & 0 & 1 & 0 & \cdots &0\\
0 & 1 & \cdots & 0 & 1  & 1 & \cdots &0\\
\vdots  & \vdots  & \ddots & \vdots & \vdots & \vdots & \ddots &0\\
0 & 0 & \cdots & 1 & 1 & 0 & \cdots & 1
\end{bmatrix}.
\end{equation}

%Now we prove that $\chi$ is indeed a characteristic function.
\begin{lemma}
The function $\chi$ defined in \eqref{eq:charforKleinbottle} is a characteristic function for $\mathcal{F}(I^{n})$.
\end{lemma}
\begin{proof}
Let $v$ be the vertex of $I^{n}$. Consider the subcollection $\mathcal{F}(v)=\{F\in \mathcal{F}(I^{n}) : v\in F \}$ of facets of $\mathcal{F}(I^{n})$. 
Then,  the image of this subcollection under the characteristic function defined in \eqref{eq:charforKleinbottle} is given by the following expression %it is easy to see that 
\[\chi(\mathcal{F}(v)) = \begin{cases}
\{e_{2},\dots,e_{n},\sum_{i=1}^{n}e_{i}\}, & \text{if $v\in F_{n+1}$,}\\
\{e_{1},\dots,e_{n}\}, & \text{otherwise}.
\end{cases}\]
Clearly, in any case $\chi(\mathcal{F}(v))$ forms a basis for $\mathbb{Z}_{2}^{n}$. This concludes the proof.
\end{proof}

We follow \cite{zbMATH06673337} for basic information about the real Bott manifolds.
Given a strictly upper triangular binary matrix (i.e., a matrix whose entries are $0$ or $1$), a real Bott manifold can be constructed as the quotient of the $n$-dimensional torus by a free action of $\mathbb{Z}_{2}^{n}$. 

\begin{definition}
A binary square matrix $A$ is said to be a Bott matrix if there exists a permutation matrix $P$ and a strictly upper triangular binary matrix $B$
such that $A=PBP^{-1}$.
\end{definition}

Let $z\in S^1$ and $a\in \{0,1\}$. Define the notation
\[z(a) := \begin{cases}
z& \text{if $a=0$,}\\
\bar{z} & \text{if $a=1$}.
\end{cases}\]
Let $A^{i}_{j}$ be the $(i,j)$ entry of a Bott matrix $A$. 
For $1\leq i \leq n$ define the involution $a_{i}$ on $(S^{1})^{n}$ as follows:
\begin{equation}\label{bottino}
 a_{i}((z_{1},\dots,z_{n}))=(z_{1}(A^{i}_{1}),\dots, z_{i-1}(A^{i}_{i-1}),-z_{i},z_{i+1}(A^{i}_{i+1}),\dots, z_{n}(A^{i}_{n})).   
\end{equation}
Note that these involutions commute with each other and generate a multiplicative group $\mathbb{Z}_{2}^{n}$, which we denote by $G(A)$. 
Moreover, it can be observed that the action of $G(A)$ on $(S^{1})^{n}$ is free (see \cite[Lemma 2.1]{zbMATH06673337}).
\begin{definition}
A real Bott manifold associated with the Bott matrix $A$, denoted by $M(A)$, is defined as the quotient \[M(A):=\frac{(S^{1})^{n}}{G(A)}.\]    
\end{definition}

Recall that the $n$-dimensional real Bott manifolds are small covers over $n$-cube, and the corresponding characteristic function is determined by the Bott matrix. 
Let $B=[b_{i,j}]$ be the Bott matrix and $F_{1},\dots, F_{n},F_{n+1},\dots,F_{2n}$ are the facets of $I^n$.
Then, the corresponding characteristic function is:   
\begin{equation}\label{eq4}
 \chi(F) = \begin{cases}
e_{i} & \text{if $F=F_{i}$ for $1\leq i \leq n$,}\\
e_{i}+\sum_{k=i+1}^{n}b_{i,k}e_{k} & \text{if $F=F_{n+i}$ for $1\leq i \leq n-1$,}\\
e_{n} & \text{if $F=F_{2n}$}. 
\end{cases}  
\end{equation}
It can be seen that the matrix of this characteristic function is given by 
$
\begin{bmatrix}
  \begin{matrix}
  \mathbf{I}_{n} & \rvline &  \mathbf{I}_{n}+B^{T} 
  \end{matrix}
\end{bmatrix}$,
where $\mathbf{I}_{n}$ is the block of $n\times n$ identity matrix and $B^T$ is the transpose of $B$.
%Now we prove that $K_{n}$ is indeed a real Bott manifold.
\begin{theorem}\label{rbm}
The $n$-dimensional Klein bottle $K_{n}$ is a real Bott manifold corresponding to the Bott matrix 
\begin{equation}\label{bottmat}
B = 
\begin{bmatrix}
0 & 1 & \cdots & 1 \\
0 & 0 & \cdots & 0 \\
\vdots  & \vdots  & \ddots & \vdots  \\
0 & 0 & \cdots & 0 
\end{bmatrix}.   
\end{equation}
In particular, $K_{n}$ is homeomorphic to the small cover $X(I^{n}, \chi)$, where $\chi$ is defined in the \eqref{eq:charforKleinbottle}.
\end{theorem}
\begin{proof}
By the quotient construction, we have 
$M(B)= \frac{(S^{1})^{n}}{G(B)}$, where $G(B)=<a_{1},\dots, a_{n}>$, $a_{1}((z_{1},z_{2},\dots,z_{n}))=(-z_{1},\bar{z}_{2},\dots,\bar{z}_{n})$, and 
$a_{i}((z_{1},\dots,z_{n}))= (z_{1},\dots, -z_{i},\dots,z_{n})$, for $2\leq i\leq n.$
Hence, \[M(B)\cong \frac{S^{1}\times S^{n-1}}{<a_{1}>\times <a_{2},\dots,a_{n}>}.\]
Equivalently, $M(B)\cong S^{1}\times_{\mathbb{Z}_{2}} (\mathbb{R}P^{1})^{n-1},$
where the action of $\mathbb{Z}_{2}$ is given by an involution  $a_{1}((z_{1},[z_{2},\dots,z_{n}]))=(-z_{1},[\bar{z}_{2},\dots,\bar{z}_{n}])$. 
Consequently, $M(B)$ is homeomorphic to $K_{n}$.

In the case of $K_n$, the characteristic matrix given by \eqref{eq2} coincides with the characteristic matrix generated by the Bott matrix $B$. Thus, $K_{n}$ is the small cover $X(I^{n},\chi)$.
\end{proof}

We recall that the mod-$2$ Betti numbers of small covers were computed by M. Davis and T. Jaunuszkeiwicz in \cite[Theorem 3.1]{zbMATH04212906}. 
%In particular the generating function for these Betti numbers is same as the $h$ polynomial of the quotient polytope. 
Ishida \cite{zbMATH05935795}, gave a formula for rational Betti numbers of real Bott manifolds. 
Let $A=[A^{i}_{j}]$ be a Bott matrix and $M(A)$ be the corresponding real Bott manifold
of $M(A)$. 

\begin{theorem}[{\cite[Lemma 2.1]{zbMATH05935795}}\label{Qbettofbottmfds}]
Let $\beta_{i}(M(A),\mathbb{Q})$ be the $i$-th rational Betti number of $M(A)$ and $A_j$ denotes the $j$-th column of $A$. Then
\[\beta_{i}(M(A),\mathbb{Q})=|\{J\subseteq [n] : |J|=i \text{ and } \sum_{j\in J}A_{j}=0\}|.\]
\end{theorem}
The following result is a consequence of \Cref{Qbettofbottmfds}.
\begin{proposition}\label{Prop: Betti numbers Kn}
Let $\beta_{i}(K_{n},\mathbb{Q})$ be the $i$-th rational Betti number of $K_{n}$. Then \[\beta_{i}(K_{n},\mathbb{Q}) = \begin{cases}
	\binom{n-1}{i} & \text{if $i$ is even  }\\
    \binom{n-1}{i-1} & \text{if $i$ is odd}.\\
	\end{cases}\]
\end{proposition}
\begin{proof}
Recall that $K_{n}\cong M(B)$ where $B$ is given by \Cref{bottmat}.  For each $i\in [n]$ consider the following collection of subsets of $[n]$ \[\mathcal{S}_{B}(i)= \{J\subseteq [n] : |J|=i \text{ and } \sum_{j\in J}B_{j}=0\} .\] Note that, if $i$ is even , then  $\mathcal{S}_{B}(i)=\{J\subseteq [n] : |J|=i \text{ and } 1\notin J\}$ and
if $i$ is odd  then $\mathcal{S}_{B}(i)=\{J\subseteq [n] : |J|=i-1 \text{ and } 1\in J\}.$
The proposition follows by counting the elements of $\mathcal{S}_{B}(i)$. 
\end{proof}

\begin{remark}
We make two observations at this point. 
\begin{enumerate}
    \item If $n$ is odd  then  $\beta_{n}(K_{n},\Q)=1$, i.e., $K_{n}$ is orientable.
    \item For all values of $n$, we have $\sum_{i=1}^{n}\beta_{i}(K_{n},\mathbb{Q})=2^{n-1}$.
\end{enumerate}
\end{remark}

In the remaining section, we will provide an alternative proof of \Cref{Prop: Betti numbers Kn} using a formula given by Suciu and Trevisan  \cite{ST}.
Let $P$ be an $n$-dimensional, simple polytope with $m$ facets and $K$ be the simplicial complex dual to the boundary $\partial(P)$.
Let $\chi$ be an $n\times m$ characteristic matrix of $P$ with entries from $\mathbb{Z}_{2}$.
For a subset $T\subseteq [n]$, define $\chi_{T}:=\sum_{i\in T}\chi_{i}$, where $\chi_{i}$ is the $i$-th row of $\chi$.
Let $K_{\chi,T}$ be the induced subcomplex of $K$ on the vertex set \[\mathrm{supp}(\chi_{T}) :=\{i\in [m] \mid \text{$i$-th entry of $\chi_{T}$ is nonzero} \}.\] 
\begin{theorem}[{\cite{ST, suciu2013rational}\label{thm:ST}}]
Let $\beta_{i}$ be the $i$-th rational Betti number of a small cover $X(P,\chi)$. 
Then \[\beta_{i}=\sum_{T\subseteq [n]}\Tilde{\beta}_{i-1}(K_{\chi, T},\mathbb{Q}),\]
where $\Tilde{\beta}_{i-1}(K_{\chi, T},\mathbb{Q})$ is the $(i-1)$-th reduced rational Betti number of $K_{\chi,T}$.
\end{theorem}

\begin{lemma}\label{supp}\
Let $\chi$ be the characteristic function for $I^{n}$ defined in \eqref{eq:charforKleinbottle}. For $T\subseteq [n]$ we have
\[
|\mathrm{supp}(\chi_{T})|=\begin{cases}
2|T| & \text{if $|T|$ is even  and $1\notin T$}\\
2|T|-1 & \text{if $|T|$ is even  and $1\in T$}\\
2|T| & \text{if $|T|$ is odd  and $1\in T$}\\
2|T|+1 & \text{if $|T|$ is odd  and $1\notin T$.}\\
\end{cases}
\]
\end{lemma}
\begin{proof}
		Let $\chi_{i}$ be the $i$-th row of the characteristic matrix of $\chi$ (see \eqref{eq2}).
		Note that for $2\leq i \leq n$, $\chi_{i}$ contains exactly three $1$'s and $\chi_{1}$ contains exactly two $1$'s. 
		Moreover, the $i$-th and $(n+i)$-th colomn are same for $2\leq i \leq n$.
		For a subset $T\subset [n]$ and $i\in T\setminus \{1\}$, the entry $1$ occurs as the $i$-th and $(n+i)$-th coordinate of vector $\chi_{T}$. 
		
		Suppose $|T|$ is odd  and $1\notin T$. 
		Then the entry $1$ occurs in $\chi_{T}$ at the $(n+1)$-st position. 
		Note that $T\subseteq [n]\setminus \{1\}$. 
		Therefore, for $i\in T$, the entry $1$ is placed at $i$-th, $(n+i)$-th and $(n+1)$-st positions. 
		Thus, $1$ occurs $2|T|+1$ many times in $\chi_{T}$. 
		
		Now assume that $|T|$ is odd and $1\in T$. 
		Note that for $i\in T\setminus \{1\}$, the entry $1$ already occurred at the $i$-th and $(n+i)$-th position.
		So $\chi_{T}$ contains $2(|T|-1)$ such $1$'s. 
		Two more $1$'s are added one of them at the $1$-st and the other at the $(n+1)$-st position. 
		Thus, $1$ occurs $2(|T|-1)+2=2|T|$ many times in $\chi_{T}$.
		
		Suppose $|T|$ is even with $1\notin T$.
		Then for each $i\in T$, the entry $1$ will occur at $i$-th and $(n+i)$-th position but not at the $(n+1)$-st position. 
		Thus, $1$ occurs in $\chi_{T}$ exactly $2|T|$ times.
		
		We now assume that $|T|$ is even and $1\in T$.
		Then for each $i\in T\setminus \{1\}$, $1$ occurs at the $i$-th and $(n+i)$-th position  but does't occurs at the $(n+1)$-st position of $\chi_{T}$. Thus, there are $2(|T|-1)$ such $1$'s in $\chi_{T}$. Since $1\in T$, one more extra $1$ gets added in $\chi_{T}$. Therefore, there are $2(|T|-1)+1=2|T|-1$ many $1$'s occurs in a vector $\chi_{T}$. This concludes the proof.
	\end{proof}
	
We now determine the homotopy types of the subcomplexes $K_{\chi, T}$ for any $T\subseteq [n]$. We denote the homotopy equivalence by the notation $\simeq$.	
\begin{lemma}\label{supp1}
		Let $K_{\chi,T}$ be the subcomplex of $K$ defined above. Then,
		\[K_{\chi,T} \simeq \begin{cases}
			S^{|T|-1}& \text{ if $|\mathrm{supp}(\chi_{T})|$ is even,}\\
			\{\star\} & \text{ if $|\mathrm{supp}(\chi_{T})|$ is odd . }
		\end{cases}\]
		
	\end{lemma}
	\begin{proof}
		Suppose $|\mathrm{supp}(\chi_{T})|$ is even. 
		Then it follows from  \Cref{supp} that, either $|T|$ is even and $1\notin T$ or $|T|$ is odd and $1\in T$. 
		
		Consider the first possibility of $|T|$ being even and $1\notin T$.
		Let $K$ be the boundary of the cross polytope of dimension $n$.
		Observe that for each $1\leq i \leq n$ the vertex $i$ of $K$ is antipodal to the  vertex $n+i$.
		Note that $T\subseteq [n]\setminus \{1\}$. Therefore, for each $i\in T$, $1$ occurs at the $i$-th and $(n+i)$-th position of vector $\chi_{T}$. 
		Consequently, $K_{\chi, T}$ can be obtained from $K$ by removing stars of those antipodal vertices which do not belong to $\mathrm{supp}(\chi_{T})$. 
		Therefore, the subcomplex $K_{\chi, T}$ is the boundary of $|T|$-dimensional cross polytope. 
		In other words, $K_{\chi,T}\simeq S^{|T|-1}.$
		
		Now consider the other possibility of $|T|$ being odd and $1\in T$. 
		Clearly, $1$ occurs at the $1$-st and $(n+1)$-th position of $\chi_{T}$. 
		Recall that the vertices in $\mathrm{supp}(\chi_{T})\setminus \{n\}$ are antipodal. 
		Therefore, for each $i\in T$, $1$ occurs at the $i$-th and $(n+i)$-th position of vector $\chi_{T}$. 
		Then it is clear that $K_{\chi, T}$ is obtained from $K$ by removing stars of antipodal vertices which do not belong to $\mathrm{supp}(\chi_{T})$.
		Therefore, $K_{\chi, T}$ is the boundary of $|T|$-dimensional cross polytope. 
		This gives $K_{\chi,T}\simeq S^{|T|-1}.$

		Now assume that $|\mathrm{supp}(\chi_{T})|$ is odd . 
		Then by \Cref{supp}, either $|T|$ is even  and $1\in T$ or $|T|$ is odd  and $1\notin T$. 
		Consider the first possibility that $|T|$ is even and $1\in T$.
		Therefore, $1$ occurs at the $1$-st position but not at the $(n+1)$-th position of $\chi_{T}$.
		Since the vertices in $\mathrm{supp}(\chi_{T})\setminus \{1\}$ are antipodal,
		one can check that
		$K_{\chi,T}\simeq S^{|T|-1}\setminus star(\{n+1\})$. 
		Clearly, $K_{\chi,T}\simeq \{\star\}$.
		Similarly, in the second possibility, we get that $K_{\chi,T}\simeq S^{|T|-1}\setminus star(\{1\})$.
		Therefore, $K_{\chi,T}\simeq \{\star\}$.
		This proves the lemma in the second case.
	\end{proof}
	
	We now compute the rational Betti numbers of $K_{n}$ using the Suciu-Trevisan formula. 
	
\begin{theorem}\label{thm:QBettiKn}
		Let $\beta_{i}$ be the $i$-th rational Betti number of $K_{n}$. Then 
		$$
		\beta_{i} = \begin{cases}
			\binom{n-1}{i} & \text{if $i$ is even  }\\
			\binom{n-1}{i-1} & \text{if $i$ is odd}.\\
		\end{cases}
		$$
		
	\end{theorem}
	\begin{proof}
		It follows from the \Cref{supp} and  \Cref{supp1} that the reduced rational homology of $K_{\chi, T}$ is $\Tilde{H}_{i-1}(K_{\chi, T}, \mathbb{Q}) \simeq \mathbb{Q}$ if and only if 
		\begin{enumerate}
			\item $|T|=i$ is even  and $1\notin T$.
			\item $|T|=i$ is odd  and $1\in T$.
		\end{enumerate}
		Now we can use Suciu-Trevisan formula to compute the Betti numbers of $K_{n}$.
		If $i$ is even, then the corresponding Betti number is a number of $i$-element subsets $[n]$ not containing $1$, and if  $i$ is odd, then the corresponding Betti number is the number of $i$-element subsets $[n]$ containing $1$. This proves the theorem.
	\end{proof}

	\begin{remark}
		Observe that, if $n$ is odd  then $\chi_{[n]}=(1,1,\dots,1)$. Therefore, $K_{\chi,[n]}=K$. In particular, $\beta_{2k+1}(K_{2k+1})=1$ for all $k$. Consequently, for each $k$, $K_{2k+1}$ is orientable.
	\end{remark}

Now we prove some properties of $K_n$, which follow from its real Bott structure.
Recall that a closed manifold $M$ of dimension $2n$ is \emph{cohomologically symplectic} if there exists a cohomology class $\alpha\in H^*(M)$ such that $\alpha^n\neq 0$.

\begin{proposition} \label{prop:cohosymplect}
Let $K_n$ be the $n$-dimensional Klein bottle. Then we have the following:
\begin{enumerate}
\item $K_{n}$ is orientable if and only if $n$ is odd ,
\item for no value of $n\geq 1$ the manifold $K_{n}$ is cohomologically symplectic.
\end{enumerate}
\end{proposition}
\begin{proof}
It was shown in the first part of \cite[Lemma 2.2]{zbMATH06673337} that the real Bott manifold $M(A)$ corresponding to a Bott matrix $A=[A^i_j]$ is orientable if and only if all row sums of $A$ are zero in $\mathbb{Z}_{2}$. 
Recall that the Bott matrix $B$ associated with $K_{n}$ is given by \Cref{bottmat}.
All row sums of $B$ are zero if and only if $n$ is odd.
This proves the first of the lemma. 

The second part of \cite[Lemma 2.2]{zbMATH06673337} says, $M(A)$ admits a symplectic form if and only if $|\{k : A_{k}=A_{j}\}|$ is even for every $1\leq j\leq n$.
Let $B_{i}$ is the $i$-th column of $B$. 
For each $j\in [n]$, consider the collection $B(j)=\{k\in [n] : B_{k}=B_{j}\}$.
Note that $|B(1)|=1$. Therefore, $K_n$ never admits a symplectic form.
\end{proof}

\begin{remark}
The first part of the above lemma also follows from \cite[Proposition 3.1]{zbMATH07134923}.
\end{remark}

%Let $M(A)$ be the real Bott manifold corresponding to a Bott matrix $A$. 
The rational cohomology ring $H^{\ast}(M(A),\mathbb{Q})$ was computed by Choi and Park  in \cite{choi2017multiplication}. 
They showed that $H^{\ast}(M(A),\mathbb{Q})$ is completely determined by the binary matroid of $A$. 
We refer the reader to \cite[Section 4]{choi2017multiplication} for more details. 

Let $A$ be a Bott matrix and $E=\{A_j : 1\leq j\leq n\}$ be the set of its columns. 
A subset $C\subseteq E$ is said to be \emph{minimally dependent} if every proper subset of $C$ is linearly independent.
We consider the collection $\mathcal{C}=\{C : C\subseteq E \text{ is minimally dependent}\}$.
The matroid $T(A)=(E,\mathcal{C})$ is called a \emph{binary matroid} associated with $A$ and the elements $C\in \mathcal{C}$ are called \emph{circuits}. 

\begin{theorem}[{\cite[Proposition 4.3]{choi2017multiplication}}]
Let $x_C$ be the formal symbol for the cohomology class corresponding to a circuit $C$. Then \[H^{\ast}(M(A),\mathbb{Q})\cong \frac{ \mathbb{Q}<x_C : C\in \mathcal{C}>}{\sim},\] where the relations are given as follows: \[x_{C}x_{C'} = \begin{cases}
	(-1)^{|C||C'|}x_{C}x_{C'} & \text{if $C\cap C'=\emptyset$ }\\
    0 & \text{if $C\cap C'\neq \emptyset$ },\\
	\end{cases}\] with $\deg(x_C)=|C|$. 
\end{theorem}

The binary matroid corresponding to the Bott matrix of $K_n$ is  \[\mathcal{C}=\{\{1\},\{i,j\} : 2\leq i<j\leq n\}.\]

Let $Y$ be the formal symbol of degree $1$ cohomology class corresponding to the singleton set $\{1\}$ and for each $\{i,j\}\in \mathcal{C}$, let $X_{ij}$ be the formal symbol of degree $2$ cohomology class. 
Then we have \[H^{\ast}(K_n,\mathbb{Q})\cong \frac{ \mathbb{Q}[Y, X_{ij} : 2\leq i<j\leq n]}{\sim},\] where the following relations hold for $2\leq i<j\leq n$ and $2\leq k< l\leq n$.
\begin{enumerate}
\item $Y^2=X_{ij}^2=0$,
\item $YX_{ij}=X_{ij}Y$,
\item $X_{ij}X_{kl}=X_{kl}X_{ij}$ if  $\{i,j\}\cap \{k,l\}=\emptyset$,
\item $X_{ij}X_{kl}=0$ if $\{i,j\}\cap \{k,l\}\neq\emptyset$.
\end{enumerate}

\section{The case of other two long genetic codes}\label{sec3}
In this section, we define certain characteristic functions on the facets of $P_{5}\times I^{n-2}$ and $P_{6}\times I^{n-2}$, where $P_n$ is the $n$-gon.
We show that the corresponding small covers $X(P_{5}\times I^{n-2})$ and $X(P_{6}\times I^{n-2})$ are homeomorphic to the planar polygon spaces associated with the genetic codes $\la\{1,2,\dots,n-2,n,n+3\}\ra$ and $\la\{1,2,\dots,n-2,n+1,n+3\}\ra$ respectively. 
\subsection{The small cover \texorpdfstring{$X(P_{5}\times I^{n-2},\chi)$}{Lg}}
We refer reader to \cite{article} for the following definition and remark. 
\begin{definition}
Let $P$ and $P'$ are two convex polytopes of dimension $d$ and $d'$, both containing the origin. Then their direct sum is a $(d+d')$-dimensional polytope  \[P\oplus P'=\emph{conv}(\{(p,0)\in \mathbb{R}^{d+d'} : p\in P\}\cup \{(0,p')\in \mathbb{R}^{d+d'} : p' \in P'\}).\]
\end{definition}

\begin{remark}\label{dp}
Let $P^{\triangle}$ and $P'^{\triangle}$ be the dual polytopes of $P$ and $P'$, respectively, containing the origin. Then their direct sum and product is related as 
$P\times P'= (P^{\triangle}\oplus P'^{\triangle})^{\triangle}$. 
In particular, if $P_{m}$ is the $m$-gon then $(P_{m}\times I^{n-2})^{\triangle}= P_{m}\oplus (I^{n-2})^{\triangle}.$
\end{remark}

To construct the characteristic function over $P_{5}\times I^{n-2}$, we give a specific labeling for the facets of $P_{5}\times I^{n-2}$ as follows :
for each $1\leq i \leq n-2$, we set
\begin{itemize}
    \item  $F_{i}=P_{5}\times I \times \dots \times \{-1\}\times \dots \times I,$ where $\{-1\}$ is at the $i$-th position.
    \item $F_{n+i}=P_{5}\times I \times \dots \times \{1\}\times \dots \times I,$ where $\{1\}$ is at the $i$-th position.
    \item For $1\leq i\leq 5$, let $E_{i}$ be the $i$-th side of $P_{5}$. We set \[F_{n-1}=E_{1}\times I^{n-2}, F_{n}=E_{2}\times I^{n-2},  F_{2n-1}=E_{3}\times I^{n-2},\] \[F_{2n}=E_{4}\times I^{n-2}, F_{2n+1}=E_{5}\times I^{n-2}. \]
\end{itemize}
Let $\mathcal{F}(P_{5}\times I^{n-2})$ be the collection of facets of $P_{5}\times I^{n-2}$.
We define a function $\chi : \mathcal{F}(P_{5}\times I^{n-2})\longrightarrow \mathbb{Z}_{2}^{n}$ by
\begin{equation}\label{eq3}
  \chi(F) = \begin{cases}
e_{i} & \text{if $F=F_{i}$ and $F=F_{n+i}$, $1\leq i \leq n$ }\\
\sum_{i=1}^{n}e_{i} & \text{if $F=E_{5}\times I^{n-2}$}.\\
\end{cases}  
\end{equation}

\begin{lemma}
The function $\chi$ is a characteristic function for $P_{5}\times I^{n-2}$.
\end{lemma}
\begin{proof}
Observe that 
\[\chi(\mathcal{F}(v)) = \begin{cases}
\{e_{1},\dots,e_{n-1},\sum_{i=1}^{n}e_{i}\} & \text{if $v\in F_{2n+1}$,}\\
\{e_{1},\dots,e_{n}\} & \text{otherwise}.
\end{cases}\] Therefore, for any vertex, $\chi(\mathcal{F}(v))$ forms a basis of $\mathbb{Z}_{2}^{n}$. Consequently, $\chi$ is the characteristic function. 
\end{proof}
It is clear that the $n\times (2n+1)$-matrix of $\chi$ is
\begin{equation}\label{eq:charmatrixp5I}
 \chi = 
\begin{bmatrix}
1 & 0 & \cdots & 0 & 1 & 0 & \cdots &0 & 1\\
0 & 1 & \cdots & 0 & 0  & 1 & \cdots &0 & 1\\
\vdots  & \vdots  & \ddots & \vdots & \vdots & \vdots & \ddots &0 & \vdots\\
0 & 0 & \cdots & 1 & 0 & 0 & \cdots & 1 & 1
\end{bmatrix}.   
\end{equation}

\begin{theorem}\label{thm:small-cover-sructure-P5I}
Let $\alpha$ be a length vector whose genetic code is $\la\{1,2,\dots,n-2,n,n+3\}\ra$ and let $\chi$ be a characteristic function defined as in \eqref{eq3}. Then $X(P_{5}\times I^{n-2}, \chi)\cong \mbalpha$.
\end{theorem}
\begin{proof}
We will show that $X(P_5\times I^{n-2},\chi)$ is homeomorphic to the chain space $\mathrm{Ch}(\beta)$, where the short code of $\beta$ is $\la\{1,2,\dots,n-2,n,n+2\}\ra$. (The length vector $\beta$ is of size $n+2$). Recall that $\mathrm{Ch}(\beta)$ is a small cover over the simple polytope $P(\beta)$ (see \cite[Page 9]{daundkar2021moment}) and the corresponding characteristic function $\chi'$ is obtained using the description given on \cite[Page 15]{daundkar2021moment}. It follows from \cite[Lemma 4.5]{daundkar2021moment} and \cite[Theorem 3]{trifreeptp} (or see \cite[Theorem 5.1]{daundkar2021moment}) that $P(\beta)\cong P_5\times I^{n-2}$.  Our next task is to show that $\chi$ coincides with the characteristic function of $P(\beta)$.
Now observe that the one-one correspondence between the facets of $P(\beta)$ and $P_5\times I^{n-2}$ can be given as follows: for $1\leq i\leq n$
\[F_i \longrightarrow F_i, ~~ \bar{F}_{i}\longrightarrow F_{n+i}, ~~ F_{2n+1}\longrightarrow E_5\times I^{n-2}.\]
Now we define $\chi'$ over facets of $P(\beta)$ following the description given on  \cite[Page 15]{daundkar2021moment}.
\[\chi'(F_i)=e_i= \chi'(\bar{F}_{i})  ~ \text{for}~ 1\leq i\leq n, ~\text{since}~ \{i,n+2\}~ \text{is short}~\text{and},\]
\[\chi'(F_{2n+1})=\sum_{i=1}^{n}e_i,~\text{since}~ \{n+1,n+2\}~ \text{is long}.\]
One can see that $\chi'$ coincides with $\chi$ defined in \eqref{eq3}. Now it follows from \cite[Remark 3.4]{daundkar2021moment} that, $\mathrm{Ch}(\alpha)\cong \mbalpha$. 
This concludes the result. 
\end{proof}

We will now compute rational Betti numbers of $\mbalpha$ with $\alpha$  having the genetic code $\la\{1,2,\dots,n-2,n,n+3\}\ra$.

\begin{lemma}
Let $\chi$ be a characteristic function defined in \eqref{eq3}. Then for any subset $T\subseteq [n]$ we have 
\[|\mathrm{supp}(\chi_{T})| = \begin{cases}
2|T| & \text{if $|T|$ is even}\\
2|T|+1 & \text{if $|T|$ is odd}.\\
\end{cases}\]
\end{lemma}
\begin{proof}
Observe that each row of the characteristic matrix given in \eqref{eq:charmatrixp5I} contains three $1$'s, and for each  $1\leq i\leq n$, the $i$-th and $(n+i)$-th column coincides.
One can see that, $1$ occurs at the $i$-th and $(n+i)$-th position of the vector $\chi_{T}$, for $i\in T$ and $1\leq i\leq n$.
Moreover, if $|T|$ is odd  then $1$ occurs in $\chi_{T}$ at the $(2n+1)$-st position as well. 
Thus, $1$ occurs $2|T|+1$ many times in $\chi_{T}$.
Now, suppose $|T|$ is even. 
Then, by the description of the characteristic matrix \eqref{eq:charmatrixp5I}, one can observe that $1$ occurs at $i$-th and $(n+i)$-th position of $\chi_{T}$ but doesn't occur at the $(2n+1)$-st position. 
Thus, in this case $1$ occurs in $\chi_{T}$ exactly $2|T|$ times.
\end{proof}

\begin{lemma}\label{p51p52}
Let $\chi$ be a characteristic function defined in \eqref{eq3}. Then for $T\subseteq [n]$, we have the following homotopy types.
\begin{enumerate}
\item Suppose $\{n-1,n\}\subseteq T$. Then
$K_{\chi, T}\simeq\begin{cases}
S^{|T|-1} & \text{$|T|$ is odd}\\
\{\star\} & \text{$|T|$ is even .}\\
\end{cases}$
\vspace{.2cm}
\item Suppose $\{n-1,n\}\nsubseteq T$. Then 
$K_{\chi, T}\simeq\begin{cases}
S^{|T|-1} & \text{$|T|$ is even}\\
\{\star\} & \text{$|T|$ is odd .}\\
\end{cases}$
\end{enumerate}
\end{lemma}

\begin{proof}
(1) Suppose $\{n-1,n\}\subseteq T$.
We first assume that $|T|$ is odd . Then, \[\{n-1,n,2n-1,2n,2n+1\}\subseteq \mathrm{supp}(\chi_{T}).\] 
Since the above set forms a vertex set of $P_5$, we have $P_{5}\subseteq K_{\chi, T}$. 
Therefore, the antipodal vertices which does not belongs to \[([n-2] \cup \{n+i : i\in [n-2]\})\cap \mathrm{supp}(\chi_{T})\] are removed from  $K_{\chi, T}$. 
Since we have K$\cong \partial(P_{5}\oplus (I^{n-2})^{\triangle})$ and  $K_{\chi, T}\simeq \partial(P_{5}\oplus (I^{|T|-2})^{\triangle})$, we have $K_{\chi, T}\simeq S^{|T|-1}.$
If $|T|$ is even, then $2n+1\notin \mathrm{supp}(\chi_{T})$. 
This gives, $K_{\chi, T}\simeq S^{|T|-1}\setminus star(\{2n+1\}).$
Consequently, $K_{\chi, T}\simeq \{\star\}.$

(2) Suppose $\{n-1,n\}\nsubseteq T$, Assume that $|T|$ is even. 
Then we have, \[\{n-1,n,2n-1,2n, 2n+1\}\nsubseteq \mathrm{supp}(\chi_{T}).\] 
Thus, $P_{5}\nsubseteq K_{\chi,T}$. 
It follows from \Cref{dp} that $K\simeq \partial(P_{5}\oplus (I^{n-2})^{\triangle}).$ 
Therefore, we have $K_{\chi,T}\simeq \partial((I^{|T|})^{\triangle}).$
Now assume that $|T|$ is odd . 
This gives us $\{n-1,n,2n-1,2n\}\nsubseteq \mathrm{supp}(\chi_{T})$ and
$2n+1\in \mathrm{supp}(\chi_{T})$. 
Note that the vertex $2n+1$ in $K$ is adjacent to all the vertices in $[n-2] \cup \{n+i : i\in [n-2]\}.$ 
Therefore, in $K_{\chi, T}$ the vertex $2n+1$ is adjacent to $([n-2] \cup \{n+i : i\in [n-2]\})\cap \mathrm{supp}(\chi_{T})$.
This gives $K_{\chi,T}$ is isomorphic to the cone over $S^{|T|-1}$ with the apex vertex $2n+1$ as $K_{\chi,T}\setminus \{2n+1\}\simeq S^{|T|-1}.$  This proves the part (2).
\end{proof}

\begin{lemma}\label{P53}
Let $\chi$ be a characteristic function defined in \eqref{eq3}. For $T\subseteq [n]$,
if the following holds 
\begin{enumerate}
    \item either $n-1\notin T$ and $n\in T$
    \item or $n-1\in T$ and $n\notin T$,
\end{enumerate}
then $K_{\chi,T}\simeq S^{|T|-1}$.
\end{lemma}
\begin{proof}
Suppose $n-1\notin T$ and $n\in T$ with $|T|$ is even . This gives us \[\{n-1,2n-1,2n+1\}\nsubseteq \mathrm{supp}(\chi_{T}) \text{ and } \{n,2n\}\subseteq \mathrm{supp}(\chi_{T}).\] Therefore, $\mathrm{supp}(\chi_{T})$ contains two antipodal vertices from $P_{5}$ and $2(|T|-1)$ vertices from $(I^{n-2})^{\triangle}$. It is easy to see that $K_{\chi,T}\simeq \partial(I\oplus I^{(|T|-1})^{\triangle}).$ Thus, $K_{\chi, T}\simeq S^{|T|-1}.$
Now assume that $|T|$ is odd . Then we have \[\{n-1,2n-1\}\nsubseteq \mathrm{supp}(\chi_{T}) \text{ and } \{n,2n,2n+1\}\subseteq \mathrm{supp}(\chi_{T}).\] 
Since $\{2n,2n+1\}$ are adjacent vertices and $n$ is antipodal to $2n$, we can collapse an edge $\{2n,2n+1\}$ to $2n$. In particular, we have $P_{5}\cap K_{\chi, T}\simeq S^{0}.$ Therefore, again we have $K_{\chi,T}\simeq \partial(I\oplus I^{(|T|-1})^{\triangle}).$ This proves the lemma in the context of the first case. Similar arguments can be used to prove the lemma in the second case. 
\end{proof}

\begin{theorem}\label{thm:QbettiP5I}
Let $\alpha$ be a length vector whose genetic code is $\la\{1,2,\dots,n-2,n,n+3\}\ra$ and
let $\beta_{i}(\mbalpha,\Q)$ be its $i$-th rational Betti number. Then
%$X(P_{5}\times I^{n-2}, \chi)$. 
\[
\beta_{i}(\mbalpha,\Q) = \begin{cases}
	2\binom{n-2}{i-1}+\binom{n-2}{i} & \text{if $i$ is even and}\\
    2\binom{n-2}{i-1}+\binom{n-2}{i-2} & \text{if $i$ is odd}.\\
	\end{cases}
\]
\end{theorem}
\begin{proof}
Using \Cref{p51p52} and \Cref{P53} we have 
$\Tilde{H}_{i-1}(K_{\chi, T}, \mathbb{Q}) \cong \mathbb{Q}$ if the following conditions holds :
\begin{enumerate}
    \item If $|T|=i$ is odd  with $\{n-1,n\}\subseteq T$.
    \item If $|T|=i$ is even  with $\{n-1,n\}\nsubseteq T$.
    \item If $|T|=i$ with $n-1\notin T$ and $n\in T$.
    \item If $|T|=i$ with $n-1\in T$ and $n\notin T$. 
\end{enumerate}
We use the Suciu-Trevisan formula to compute the rational Betti numbers of $X(P_{5}\times I^{n-2}, \chi)\cong \mbalpha$. 
If $i$ is even  then the corresponding rational Betti number is the sum of the cardinalities of $i$-element subsets of $[n]$ of type $(2)$, $(3)$ and $(4)$. 
Similarly, if $i$ is odd  then the corresponding Betti number is the sum of the cardinalities of $i$-element subsets of $[n]$ of type $(1)$, $(3)$ and $(4)$. 
\end{proof}

\begin{remark}
We observe that  $\sum_{i=1}^{n}\beta_{i}(K_{n},\mathbb{Q})=3\cdot2^{n-2}$.
\end{remark}

\subsection{Betti numbers of \texorpdfstring{$X(P_{6}\times I^{n-2},\chi)$}{Lg}  }
To construct the characteristic function over $P_{6}\times I^{n-2}$, we give a specific labeling for its facets:
\begin{itemize}
    \item For each $1\leq i \leq n-2$, we set $F_{i}=P_{6}\times I \times \dots \times \{-1\}\times \dots \times I,$ where $\{-1\}$ is at the $i$-th position.
    \item For each $1\leq i \leq n-2$, we set $F_{n+1+i}=P_{6}\times I \times \dots \times \{1\}\times \dots \times I$, where $\{1\}$ is at the $i$-th position.
    \item For $1\leq i\leq 6$, let $E_{i}$ is the $i$-th side of $P_{6}$. Then we set \[F_{n-1}=E_{1}\times I^{n-2}, \hspace{1mm} F_{n}=E_{2}\times I^{n-2},\hspace{1mm} F_{2n-1}=E_{3}\times I^{n-2},\] \[F_{2n}=E_{4}\times I^{n-2},\hspace{1mm} F_{2n+1}=E_{5}\times I^{n-2},\hspace{1mm} F_{2n+1}=E_{6}\times I^{n-2}. \]
\end{itemize}
%Let $\mathcal{F}(P_{6}\times I^{n-2})$ be the   facet collection of $P_{6}\times I^{n-2}$.
Define a function $\chi : \mathcal{F}(P_{6}\times I^{n-2})\to \mathbb{Z}_{2}^{n}$ by
\begin{equation}\label{eq:charfnp6I}
\chi(F) = \begin{cases}
e_{i} & \text{if $F=F_{i}$ and $F=F_{n+1+i}$, $1\leq i \leq n$ }\\
\sum_{i=1}^{n}e_{i} & \text{if $F=F_{n+1}$ and $F=F_{2n+2}$}.\\
\end{cases}    
\end{equation}

\begin{lemma}
The function $\chi$ is a characteristic function for $P_{6}\times I^{n-2}$.
\end{lemma}
\begin{proof}
Note that
\[\chi(\mathcal{F}(v)) = \begin{cases}
\{e_{1},\dots,e_{n-1},\sum_{i=1}^{n}e_{i}\} & \text{if either $v\in F_{n+1}$ or $v\in F_{2n+2}$,}\\
\{e_{1},\dots,e_{n}\} & \text{otherwise}.
\end{cases}.\]
\end{proof}
It is clear that the $(n\times 2n)$-matrix of $\chi$ is 
\begin{equation}\label{eq:charmatrixp6i}
  \chi = 
\begin{bmatrix}
1 & 0 & \cdots & 0 & 1 & 1 & 0 & \cdots &0 & 1\\
0 & 1 & \cdots & 0 & 1& 0  & 1 & \cdots &0 & 1\\
\vdots  & \vdots  & \ddots & \vdots & \vdots & \vdots & \vdots & \ddots &0 & \vdots\\
0 & 0 & \cdots & 1 & 1 & 0 & 0 & \cdots & 1 & 1
\end{bmatrix}.   
\end{equation}

\begin{theorem}\label{thm:small-cover-sructure-P6I}
Let $\alpha$ be a length vector whose genetic code is $\la\{1,2,\dots,n-2,n+1,n+3\}\ra$ and let $\chi$ be a characteristic function defined as in \eqref{eq:charfnp6I}. Then $X(P_{5}\times I^{n-2}, \chi)\cong \mbalpha$.
\end{theorem}
\begin{proof}
The proof is similar to that of \Cref{thm:small-cover-sructure-P5I}.
\end{proof}

\begin{lemma}
Let $\chi$ be a characteristic function defined as in \eqref{eq:charfnp6I} and $T\subseteq [n]$. Then we have 
\[|\mathrm{supp}(\chi_{T})| = \begin{cases}
2|T|, & \text{if $|T|$ is even,}\\
2|T|+2, & \text{if $|T|$ is odd}.\\
\end{cases}\]
\end{lemma}
\begin{proof}
Observe that, each row of the characteristic matrix \eqref{eq:charmatrixp6i} contains four $1$'s. Moreover, $i$-th and $(n+1+i)$-th columns coincide for $1\leq i\leq n+1$. 

For each $i\in T$ with $1\leq i\leq n$, the entry $1$ occurs at the $i$-th and $(n+1+i)$-th position.
Moreover, if $|T|$ is odd  then the entry $1$ occurs in $\chi_{T}$ at the $(n+1)$-st and $(2n+2)$-nd position. 
In particular, $1$ occurs $2|T|+2$ many times in $\chi_{T}$. 

Suppose $|T|$ is even . 
Then the entry $1$ occur at $i$-th and $(n+1+i)$-th position but it doesn't occur at the $(n+1)$-st and $(2n+2)$-nd position. 
Therefore, in this case, the entry $1$ occurs exactly $2|T|$ times.
\end{proof}
Now, we determine the homotopy type of the subcomplexes. 
As before, the computations are in two parts depending on the conditions on $T$.

\begin{lemma}\label{lemm:p6i1p6i2}
Let $\chi$ be a characteristic function defined as in \eqref{eq:charfnp6I} and $T\subseteq [n]$. Then we have the following homotopy equivalences.
\begin{enumerate}
\item Suppose $\{n-1,n\}\subseteq T$. Then 
$K_{\chi, T}\simeq\begin{cases}
S^{|T|-1} & \text{$|T|$ is odd,}\\
S^{|T|-2} & \text{$|T|$ is even.}\\
\end{cases}$
\vspace{.2cm}
\item Suppose $\{n-1,n\}\nsubseteq T$. Then 
$K_{\chi, T}\simeq\begin{cases}
S^{|T|} & \text{$|T|$ is odd,}\\
S^{|T|-1} & \text{$|T|$ is even.}\\
\end{cases}$
\end{enumerate}
\end{lemma}
\begin{proof}
(1)
Suppose $\{n-1,n\}\subseteq T$. Assume that $|T|$ is odd. 
Note that for each $i\in T$ with $1\leq i\leq n$, $1$ occurs at the $i$-th and $(n+1+i)$-th position of $\chi_{T}$. 
Since $|T|$ is odd , $1$ occurs at $(n+1)$-st and $(2n+2)$-nd position of $\chi_{T}$ as well. 
Thus we have, 
\[\{n-1, n, n+1, 2n, 2n+1, 2n+2\}\subseteq \mathrm{supp}(\chi_{T}) .\]
Since the above set forms a vertex set of $P_{6}$, we have $P_{6}\subseteq K_{\chi,T}$.
The remaining vertices of $K_{\chi,T}$ are given by $\{i : i\in T\}\cup \{n+1+i : i\in T\}$
Note that $K\cong P_{6}\oplus (I^{n-2})^{\triangle}$. 
Observe that the above vertices are from the $(I^{n-2})^{\triangle}$ factor of $K$. 
Therefore,
$K_{\chi,T}\simeq \partial(P_{6}\oplus    \bigoplus_{i\in T\cap [n-2]} I_{i}),$ where $I_{i}=I$. 
Now it is clear that $K_{\chi,T}\simeq \partial(P_{6}\oplus (I^{|T|-2})^{\triangle})\simeq S^{|T|-1}.$

Now assume that $\{n-1,n\}\subseteq T$ and $|T|$ is even . 
Therefore, $1$ does not occur at the $(n+1)$-st and $(2n+2)$-nd position of vector $\chi_{T}$. 
This gives \[\{n+1,2n+2\}\nsubseteq \mathrm{supp}(\chi_{T}), \text{ and } \{n-1, n, 2n, 2n+1\}\subseteq \mathrm{supp}(\chi_{T})\] since $\{n-1,n\}\subseteq T$. 
Clearly, we have $P_{6}\cap K_{\chi,T}\simeq S^{0}.$
Now it is easy to see that $K_{\chi,T}\simeq \partial(I\oplus \bigoplus_{i=1}^{|T|-2}I_{i}),$ where $I_{i}=I$ for all $i$. 
Therefore, $K_{\chi,T}\simeq \partial((I^{|T|-1})^{\triangle})\simeq S^{|T|-2}.$ This proves Part (1).

(2)
Suppose $\{n-1,n\}\nsubseteq T$ Assume that $|T|$ is even . 
Therefore, $1$ does not occur at the $(n+1)$-st and $(2n+2)$-nd position of vector $\chi_{T}$. 
This gives \[\{n+1,2n+2\}\nsubseteq \mathrm{supp}(\chi_{T}) \text{ and } \{n-1, n, 2n, 2n+1\}\nsubseteq \mathrm{supp}(\chi_{T}).\]
Thus, $P_{6}\nsubseteq K_{\chi,T}$.
Since $T\subseteq [n-2]$, $K_{\chi,T}\simeq \partial(\oplus_{i=1}^{|T|}I_{i}),$ where $I_{i}=I$ for all $i$.
Therefore, $K_{\chi,T}\simeq \partial((I^{|T|})^{\triangle})\simeq S^{|T|-1}.$

Now assume that $\{n-1,n\}\nsubseteq T$ and $|T|$ is odd . 
Therefore, $1$ occurs at the $(n+1)$-st and $(2n+2)$-nd position of $\chi_{T}$.
Therefore, \[\{n+1,2n+2\}\subseteq \mathrm{supp}(\chi_{T}) \text{ and } \{n-1, n, 2n, 2n+1\}\nsubseteq \mathrm{supp}(\chi_{T}).\] 
Since $T\subseteq [n-2]$, $K_{\chi,T}\simeq \partial(I\oplus \bigoplus_{i=1}^{|T|}I_{i}),$ where $I_{i}=I$ for all $i$. 
Note that the first factor in the previous direct sum corresponds to $\{n+1, 2n+2\}$.
Therefore, $K_{\chi,T}\simeq \partial((I^{|T|+1})^{\triangle})\simeq S^{|T|}.$
This proves part (2).
\end{proof}

\begin{lemma}\label{P63}
Let $\chi$ be a characteristic function defined as in \eqref{eq:charfnp6I} and $T\subseteq [n]$. Then if the following conditions hold
\begin{enumerate}
    \item $n-1\notin T$ and $n\in T$ or
    \item $n-1\in T$ and $n\notin T$,
\end{enumerate}
 then we have $K_{\chi,T}\simeq S^{|T|-1}$.
\end{lemma}
\begin{proof}
Suppose $n-1\notin T$ and $n\in T$ with $|T|$ is odd . 
Therefore, the entry $1$ occurs at the $n$-th, $(2n+1)$-st, $(n+1)$-st and $(2n+2)$-nd position of $\chi_{T}$ but it doesn't occur at the $(n-1)$-st and the $(2n)$-th position.
This clearly gives \[\{n, n+1, 2n+1, 2n+2\}\subseteq \mathrm{supp}(\chi_{T}) \text{ and } \{n-1, 2n\}\nsubseteq \mathrm{supp}(\chi_{T}).\]
Since $T\setminus\{n\}\subseteq [n-2]$, $K_{\chi,T}\simeq \partial(I\oplus \bigoplus_{i=1}^{|T|-1}I_{i}),$ where $I_{i}=I$ for all $i$. 
Note that the first factor in the above direct sum corresponds to $\{n, 2n+1\}$. 
Therefore, $K_{\chi,T}\simeq \partial((I^{|T|})^{\triangle})\simeq S^{|T|-1}.$
Now suppose $n-1\notin T$ and $n\in T$ with $|T|$ is even .
Therefore, $1$  does not occur at the $(n-1)$-st, $(2n)$-th, $(n+1)$-st, $(2n+2)$-nd position of vector $\chi_{T}$ but occurs at the $n$-th and $(2n+1)$-st position. 
In particular, we have \[\{n-1, 2n, n+1, 2n+2\}\nsubseteq \mathrm{supp}(\chi_{T}) \text{ and } \{n, 2n+1\}\subseteq \mathrm{supp}(\chi_{T}).\]
Since $T\setminus\{n\}\subseteq [n-2]$, $K_{\chi,T}\simeq \partial(I\oplus_{i=1}^{|T|-1}I_{i}),$ where $I_{i}=I$ for all $i$. 
Note that the first factor in the above direct sum is corresponding to $\{n, 2n+1\}$. 
Therefore, $K_{\chi,T}\simeq \partial((I^{|T|})^{\triangle})\simeq S^{|T|-1}.$
This proves the lemma in the first case. 
Similar steps can be followed to prove the second case.
\end{proof}

\begin{theorem}\label{thm:QbettiP6I}
Let $\alpha$ be a length vector whose genetic code is $\la\{1,2,\dots,n-2,n+1,n+3\}\ra$ and let $\beta_{i}(\mbalpha, \Q)$ be the $i$-th rational Betti number of the corresponding
polygon space. Then 
\[\beta_{i}(\mbalpha,\Q) = \begin{cases}
3\binom{n-2}{i-1}+\binom{n-2}{i} & \text{if $i$ is even,}\\
3\binom{n-2}{i-1}+\binom{n-2}{i-2} & \text{if $i$ is odd}.\\
\end{cases}\]
\end{theorem}
\begin{proof}
Let $i$ be odd.
Then from \Cref{lemm:p6i1p6i2} and \Cref{P63} we have 
$\Tilde{H}_{i-1}(K_{\chi, T}, \mathbb{Q}) \cong \mathbb{Q}$ if the following conditions holds :
\begin{enumerate}
    \item If $|T|=i$  with $\{n-1,n\}\subseteq T$.
    \item If $|T|=i+1$  with $\{n-1,n\}\subseteq T$.
    \item If $|T|=i$ with $n-1\notin T$ and $n\in T$.
    \item If $|T|=i$ with $n-1\in T$ and $n\notin T$. 
\end{enumerate}
We use the Suciu-Trevisan formula to compute the rational Betti numbers of $X(P_{6}\times I^{n-2}, \chi)\cong \mbalpha$.
Note that the cardinality of type $(1)$ sets is $\binom{n-2}{i-2}$ and the cardinalities of type $(2)$, type $(3)$ and type $(4)$ sets are same and it is equal to $\binom{n-2}{i-1}$ in each case. Now theorem follows by adding these cardinalities.

Now, suppose $i$ is even. Then again we can use \Cref{lemm:p6i1p6i2} and \Cref{P63} to get the $(i-1)$-st reduced rational homology of $K_{\chi,T}$. 
We have 
$\Tilde{H}_{i-1}(K_{\chi, T}, \mathbb{Q}) \cong \mathbb{Q}$ if the following conditions holds :
\begin{enumerate}
    \item If $|T|=i$  with $\{n-1,n\}\nsubseteq T$. 
    \item If $|T|=i-1$  with $\{n-1,n\}\nsubseteq T$.
    \item If $|T|=i$ with $n-1\notin T$ and $n\in T$.
    \item If $|T|=i$ with $n-1\in T$ and $n\notin T$. 
\end{enumerate}
Note that the cardinality of type $(1)$ sets is $\binom{n-2}{i}$ and the cardinalities of type $(2)$, type $(3)$, type $(4)$ sets are same and it is equal to $\binom{n-2}{i-1}$ in each case. 
This proves the theorem. 
\end{proof}

\begin{remark}
Note that
    $\sum_{i=1}^{n}\beta_{i}(K_{n},\mathbb{Q})=4\cdot2^{n-2}.$
\end{remark}

%\begin{example}
%The following table contains first five Betti numbers of $X(P_{6}\times I^{n-2}, \chi))$ up to the dimension $5$.

%\begin{table}[H]

%\centering\begin{tabular}{|c|c|c|c|c|c|c|}
%\hline
 %   \backslashbox{$n$}{$i$}& $0$ & 1 & 2 & 3 & 4 & 5  \\ \hline
  %  2 & 1 & 3 & 0 & 0 & 0 &   0 \\
   % \hline
   % 3 & 1 & 3 & 3 & 1 & 0 & 0   \\
    %\hline
    %4 & 1 & 3 & 7 & 5 & 0 &   0 \\
    %\hline
    %5 & 1 & 3 & 12 & 12 & 3 & 1   \\
    %\hline
%\end{tabular}
%\caption{$\beta_{i}(X(P_{6}\times I^{n-2}, \chi))$.}
%\label{tab: 1}
%\end{table}
%\end{example}

%\noindent \textbf{Acknowledgement}
%The first author thanks NBHM for the support through the grant 0204/10/(16)/2023/R\&D-II/2789.

%\section{Further questions}
%We call the total space of an iterated $S^{1}$-bundle starting with a closed (non-orientable) surfaces as a \emph{Bott-type manifold}.
%It is easy to see that these manifolds are small covers over $P_i\times I^{n-2}$ where $P_i$ is an $i$-gon for an appropriate $i>2$. 
%One could explore the topological
%and combinatorial aspects of Bott-type manifolds. 
%More precisely, one can ask following questions:
%\begin{problem}
%Given $i>2$, how many Bott-type manifolds exist, up to diffeomorphism, over $P_i\times I^{n-2}$?    
%\end{problem}

%\begin{problem}
%How to characterize Bott-type manifolds, up to diffeomorphism, in terms of some combinatorial data assoicated with the fibration?   
%\end{problem}
\section{Declarations}
\subsection*{Ethical Approval}
This declaration is not applicable.
\subsection*{Consent to participate}
This declaration is not applicable.
\subsection*{Consent to publish}
This declaration is not applicable.
\subsection*{Competing interests}
The authors declare that they have no competing interests. 
\subsection*{Authors contributions}
All authors contributed equally.
\subsection*{Funding}
The first author was funded by NBHM through the grant 0204/10/(16)/2023/R\&D-II/2789. The second author was partially funded by the Infosys Foundation during this work.

\subsection*{Availability of data and materials}
Data sharing is not applicable to this article as no data sets were generated or analyzed during the current study.

\bibliographystyle{plain} 
\bibliography{references}

\begin{thebibliography}{10}

\bibitem{trifreeptp}
G.~Blind and R.~Blind.
\newblock Triangle-free polytopes with few facets.
\newblock {\em Arch. Math. (Basel)}, 58(6):599--605, 1992.

\bibitem{zbMATH06673337}
Suyoung {Choi}, Mikiya {Masuda}, and Sang-Il {Oum}.
\newblock {Classification of real Bott manifolds and acyclic digraphs}.
\newblock {\em {Trans. Am. Math. Soc.}}, 369(4):2987--3011, 2017.

\bibitem{choi2017multiplication}
Suyoung Choi and Hanchul Park.
\newblock Multiplicative structure of the cohomology ring of real toric spaces.
\newblock {\em Homology Homotopy Appl.}, 22(1):97--115, 2020.

\bibitem{daundkar2021moment}
Navnath Daundkar and Priyavrat Deshpande.
\newblock The moment polytope of the abelian polygon space.
\newblock {\em Topology Appl.}, 302:Paper No. 107834, 24, 2021.

\bibitem{Daviscring}
Donald~M. Davis.
\newblock Topological complexity of planar polygon spaces with small genetic code.
\newblock {\em Forum Math.}, 29(2):313--328, 2017.

\bibitem{zbMATH07134923}
Donald~M. {Davis}.
\newblock {An \(n\)-dimensional Klein bottle}.
\newblock {\em {Proc. R. Soc. Edinb., Sect. A, Math.}}, 149(5):1207--1221, 2019.

\bibitem{zbMATH04212906}
Michael~W. {Davis} and Tadeusz {Januszkiewicz}.
\newblock {Convex polytopes, Coxeter orbifolds and torus actions}.
\newblock {\em {Duke Math. J.}}, 62(2):417--451, 1991.

\bibitem{FS}
M.~Farber and D.~Sch\"{u}tz.
\newblock Homology of planar polygon spaces.
\newblock {\em Geom. Dedicata}, 125:75--92, 2007.

\bibitem{FAR}
Michael Farber.
\newblock {\em Invitation to topological robotics}.
\newblock Zurich Lectures in Advanced Mathematics. European Mathematical Society (EMS), Z\"{u}rich, 2008.

\bibitem{HK1}
J.-C. Hausmann and A.~Knutson.
\newblock The cohomology ring of polygon spaces.
\newblock {\em Ann. Inst. Fourier (Grenoble)}, 48(1):281--321, 1998.

\bibitem{hausmann2007geometric}
Jean-Claude Hausmann.
\newblock Geometric descriptions of polygon and chain spaces.
\newblock In {\em Topology and robotics}, volume 438 of {\em Contemp. Math.}, pages 47--57. Amer. Math. Soc., Providence, RI, 2007.

\bibitem{polspacesGrassmann}
Jean-Claude Hausmann and Allen Knutson.
\newblock Polygon spaces and {G}rassmannians.
\newblock {\em Enseign. Math. (2)}, 43(1-2):173--198, 1997.

\bibitem{Hausmann-Rodriguez}
Jean-Claude Hausmann and Eugenio Rodriguez.
\newblock The space of clouds in {E}uclidean space.
\newblock {\em Experiment. Math.}, 13(1):31--47, 2004.

\bibitem{article}
Martin Henk, J\"{u}rgen Richter-Gebert, and G\"{u}nter~M. Ziegler.
\newblock Basic properties of convex polytopes.
\newblock In {\em Handbook of discrete and computational geometry}, CRC Press Ser. Discrete Math. Appl., pages 243--270. CRC, Boca Raton, FL, 1997.

\bibitem{zbMATH05935795}
Hiroaki {Ishida}.
\newblock {Symplectic real Bott manifolds}.
\newblock {\em {Proc. Am. Math. Soc.}}, 139(8):3009--3014, 2011.

\bibitem{Kamiyamahomology2}
Yasuhiko Kamiyama and Michishige Tezuka.
\newblock Topology and geometry of equilateral polygon linkages in the {E}uclidean plane.
\newblock {\em Quart. J. Math. Oxford Ser. (2)}, 50(200):463--470, 1999.

\bibitem{Kamiyamahomology1}
Yasuhiko Kamiyama, Michishige Tezuka, and Tsuguyoshi Toma.
\newblock Homology of the configuration spaces of quasi-equilateral polygon linkages.
\newblock {\em Trans. Amer. Math. Soc.}, 350(12):4869--4896, 1998.

\bibitem{ST}
A.~I. Suciu and A.~Trevisan.
\newblock Real toric varieties and abelian covers of generalized davis-januszkiewicz spaces.
\newblock {\em Unpublished}, 2012.

\bibitem{suciu2013rational}
Alexander~I Suciu.
\newblock The rational homology of real toric manifolds.
\newblock {\em in: Mini-Workshop: Cohomology Rings and Fundamental Groups of Hyperplane Arrangements, Wonderful Compactifications, and Real Toric Varieties, Oberwolfach Reports}, 9:2972–2976, 2013 (available at arXiv:1302.2342).

\end{thebibliography}
\end{document}